\documentclass[12pt,reqno]{amsart}
\usepackage{amssymb,amsmath,amsthm,enumerate,verbatim}

\DeclareMathOperator{\Tr}{Tr}

\DeclareMathOperator{\rank}{rank}

\DeclareMathOperator{\Det}{Det}

\newcommand{\abs}[1]{\lvert#1\rvert}

\newcommand{\norm}[1]{\lVert#1\rVert}
\newcommand{\Norm}[1]{\left\lVert#1\right\rVert}

\newcommand{\bbT}{{\mathbb T}}
\newcommand{\bbR}{{\mathbb R}}
\newcommand{\bbC}{{\mathbb C}}

\newcommand{\calD}{{\mathcal D}}
\newcommand{\calH}{{\mathcal H}}
\newcommand{\calK}{{\mathcal K}}

\newcommand{\calL}{\mathcal{L}}

\newcommand{\calB}{\mathcal{B}}

\newcommand{\Sch}{\mathbf{S}}

\numberwithin{equation}{section}

\theoremstyle{plain}
\newtheorem{theorem}{\bf Theorem}[section]
\newtheorem*{theorem*}{Theorem}
\newtheorem{lemma}[theorem]{\bf Lemma}

\theoremstyle{definition}

\theoremstyle{remark}
\newtheorem*{remark*}{\bf Remark}
\newtheorem{remark}[theorem]{\bf Remark}

\newcommand{\eps}{\varepsilon}

\newcommand{\ac}{\text{\rm (ac)}}
\newcommand{\sing}{\text{\rm (sing)}}

\begin{document}

\title[Perturbation theory and two weights problem]{Spectral perturbation theory and the two weights problem}

\author{Alexander Pushnitski}
\address{Department of Mathematics, King's College London, Strand, London, WC2R~2LS, U.K.}
\email{alexander.pushnitski@kcl.ac.uk}

\author{Alexander Volberg}
\address{Department of Mathematics, Michigan State University, 
East Lansing, MI 48824, U.S.A.}
\email{volberg@math.msu.edu}



\begin{abstract}
The famous two weights problem consists in characterising 
all possible pairs of weights such that the Hardy projection 
is bounded between the corresponding weighted $L^2$ spaces. 
Koosis' theorem of 1980 gives a way to construct a certain 
class of pairs of weights. We show that Koosis' theorem
is closely related to (in fact, is a direct consequence of) 
a spectral perturbation model suggested by de Branges in 1962. 
Further, we show that de Branges' model provides an operator-valued 
version of Koosis' theorem.
\end{abstract}

\maketitle

\section{Introduction and main result}\label{sec.a}

\subsection{Introduction}\label{sec.a1}

Let $P_\pm$ be the Hardy projections in $L^2(\bbT)$ ($\bbT$ is the unit circle parameterised by $(0,2\pi)$):
\begin{equation}
(P_\pm f)(e^{i\theta})
=
\pm
\lim_{r\to1\mp0}\int_0^{2\pi}\frac{f(e^{it})}{1-re^{i(\theta-t)}}\frac{dt}{2\pi}.
\label{a1}
\end{equation}
In its simplest form, 
the two weights problem consists in the characterisation of all 
pairs of weights $v_j:\bbT\to [0,\infty)$, $j=0,1$, such that
\begin{equation}
P_+: L^2(\bbT,v_0(e^{it})dt) \mapsto L^2(\bbT, v_1(e^{it})dt)
\label{a2}
\end{equation}
is a bounded operator. (Of course, one could equally speak of $P_-$).  
If $v_1=v_0$, then the characterisation of such weights is given by 
the celebrated Muckenhoupt condition \cite{Muck}:
$$
\sup_\Delta 
\left(
\frac1{\abs{\Delta}}\int_\Delta v_0(e^{it})dt 
\cdot
\frac1{\abs{\Delta}}\int_\Delta v_0(e^{it})^{-1}dt 
\right)<\infty,
$$
where the supremum is taken over the set of all intervals $\Delta\subset(0,2\pi)$ 
and $\abs{\Delta}$ is 
the length of the interval $\Delta$.
If there is no \emph{a priori} relation between $v_0$ and $v_1$, the 
two weights problem is open, despite many years of efforts.
Some necessary and some sufficient conditions are known but 
no effective complete description of all pairs of weights $v_0$, $v_1$ was available available till recently.
The recent news at the time of writing is that the conjunction of three preprints \cite{NTV}, \cite{LSSUT}, \cite{L} proved a 
long-standing conjecture of Nazarov--Treil--Volberg (see \cite{Vo}), stating that for the Hilbert transform the
so-called two-weight $T1$ theorem is valid. However,  the conditions of $T1$ theorem are not easily 
translated (if at all) into conditions on weights.

Under these circumstances, any partial information on the problem is valuable. 
One such piece of information is Koosis' theorem \cite{Koosis1}:
\begin{theorem*}[Koosis]
For every weight $v_0\geq0$ such that
$0<v_0(e^{it})<1$ for a.e. $t\in(0,2\pi)$ and  $v_0^{-1}\in L^1(\bbT)$, 
one can find another weight $v_1$,  $0\leq v_1\leq v_0$, such that 
$\log v_1\in L^1(\bbT)$ and such that the Hardy projection $P_+$ is bounded
between the weighted spaces \eqref{a2}.
\end{theorem*}
 
Koosis' proof (see also \cite[Appendix]{NVY}) is an ingenious 
calculation, but one can argue that it has a rather \emph{ad hoc}
flavour. 
The purpose of this note is to point out that Koosis' theorem 
follows \emph{naturally} from the formalism of spectral perturbation theory
(more precisely, scattering theory)
in the form suggested by de Branges in \cite{deB}.
In fact, the statement we get in this way is more general than the 
original Koosis' theorem; we obtain an \emph{operator-valued}
analogue. 
That is, our $L^2$ spaces consist of functions on $\bbT$ with values in a Hilbert space $\calK$
and our weights are functions with values in the Schatten classes of compact 
operators in $\calK$. 

We hope that this note will attract the attention of experts to the connection 
between the two weights problem and scattering theory. We believe
that this connection is yet to be thoroughly explored.

\subsection{Preliminaries}\label{sec.a1a}
First we would like to rewrite the two weights problem in an equivalent form.
Let $f\in L^2(\bbT,v_0(e^{it})dt)$ 
and 
suppose that the weight $v_0$ vanishes
on some open set. Then the function $f$ is not defined on this open 
set, and therefore it is not clear how to define the projections $P_\pm f$ by 
\eqref{a1}. This suggests that the integration in the definition \eqref{a1} of the 
projections $P_\pm$ should be performed with respect to the weighted
measure $v_0(e^{it})dt$. 
Thus, for a weight $w_0:\bbT\to[0,\infty)$, we define the 
\emph{weighted Hardy projections $P_\pm^{(w_0)}$} by 
\begin{equation}
(P_\pm^{(w_0)}f)(e^{i\theta})
=
\pm
\lim_{r\to1\mp0}\int_0^{2\pi}\frac{w_0(e^{it})f(e^{it})}{1-re^{i(\theta-t)}}\frac{dt}{2\pi};
\label{a3}
\end{equation}
the existence of the limits will be discussed separately. 

If $v_0(e^{it})>0$ for a.e.\  $t$, then a simple argument 
with replacing $f$ by $v_0f$ 
shows that 
$P_+$  is a bounded operator  between the spaces \eqref{a2}
if and only if 
\begin{equation}
P_+^{(w_0)}: L^2(\bbT,w_0(e^{it})dt)\to L^2(\bbR,w_1(e^{it})dt)
\label{a4}
\end{equation}
is bounded, where $w_1=v_1$ and $w_0=v_0^{-1}$.
Thus, we obtain 
\begin{theorem*}[Koosis, version 2]
For every weight $w_0\geq0$ such that
$w_0(e^{it})>0$ for a.e. $t\in(0,2\pi)$ and
$w_0\in L^1(\bbT)$, 
one can find another weight $w_1\geq0$ with $w_1w_0\leq1$ 
and $\log w_1\in L^1(\bbT)$ such that the weighted Hardy projection $P_+^{(w_0)}$ 
is a bounded operator between the spaces \eqref{a4}.
\end{theorem*}

It is this second version of Koosis' theorem that we will discuss in this paper.

\subsection{Operator valued functions}\label{sec.a1b}

Let $\calK$ be a Hilbert space; the case $\dim \calK<\infty$ is not excluded,
neither it is trivial. We denote by $(\cdot,\cdot)$ the inner product in $\calK$
and by $\norm{\cdot}$ the norm in $\calK$. 
Notation $\calB(\calK)$ stands for the set of all bounded linear operators
on $\calK$ and $\Sch_p$, $1\leq p<\infty$, denotes the Schatten class of 
compact operators in $\calK$; in particular, $\Sch_1$ is the trace class. 
We denote by $\norm{\cdot}_p$ the norm in $\Sch_p$
and by $\norm{\cdot}_\calB$ the norm in $\calB(\calK)$. 
As usual, for $w\in\calB(\calK)$, notation $w\geq0$ means that $(w\chi,\chi)\geq0$
for all elements $\chi\in\calK$, and in the same way $w\leq C$, where $C$ is 
a constant, means $(w\chi,\chi)\leq C\norm{\chi}^2$ for all $\chi\in\calK$. 
For any $w\in\calB(\calK)$ such that $w\geq0$, the square root $w^{1/2}$ 
is defined via the functional calculus for self-adjoint operators.

Below we work with ``nice'' $\calK$-valued functions of the 
form
\begin{equation}
f(\mu)=\sum_i (\mu-z_i)^{-1} \chi_i,
\quad
\mu\in\bbT,
\quad
\chi_i\in\calK, 
\quad 
\abs{z_i}\not=1,
\label{a5}
\end{equation}
where the sum has finitely many terms. 
We will denote by  $\calL$ the set of all such ``nice'' functions $f$.

Let $w:\bbT\to\calB(\calK)$ be a Borel measurable function. 
Suppose that $w$ is non-negative i.e. $w(e^{it})\geq0$ for a.e.\  $t\in(0,2\pi)$,
and that $w$ satisfies
\begin{equation}
\int_0^{2\pi}(w(e^{it})\chi,\chi)\frac{dt}{2\pi} \leq C\norm{\chi}^2
\label{a4a}
\end{equation}
for some constant $C$ and all $\chi\in\calK$. 
Then for any $f\in\calL$ we can define the quasi-norm
\begin{equation}
\norm{f}_{L^2(w)}^2
=
\int_0^{2\pi} (w(e^{it})f(e^{it}),f(e^{it}))\frac{dt}{2\pi}.
\label{a4b}
\end{equation}
After taking the quotient over the subspace of functions $f$ with 
$\norm{f}_{L^2(w)}=0$, we obtain a norm on the quotient space;
the space obtained by taking the closure is, by definition, the 
weighted space $L^2(w)$. Thus, by construction, $\calL$ 
is dense in $L^2(w)$.

\subsection{Main result and discussion}\label{sec.a2}

Let $1\leq p<\infty$, and let $w_0:\bbT\to\Sch_p$ be a Borel measurable
function. 
We assume that $w_0$ is non-negative
and  satisfies
$$
\int_0^{2\pi} \norm{w_0(e^{it})}_p \frac{dt}{2\pi}<\infty;
$$
this, of course, implies \eqref{a4a}.
For convenience, we will assume that $w_0$ is normalised so that the above 
integral equals one:
\begin{equation}
\int_0^{2\pi} \norm{w_0(e^{it})}_p \frac{dt}{2\pi}=1.
\label{a6}
\end{equation}
For such weight $w_0$ and for $f\in\calL$, we define the
weighted Hardy projections $P_\pm^{(w_0)}$, as in the scalar case, by \eqref{a3}.
It is clear that for every $r\not=1$, the integrals in \eqref{a3} converge
absolutely in the norm of $\calK$.

\begin{theorem}\label{th.a1}
Let $1\leq p<\infty$, and let $w_0:\bbT\to\Sch_p$ 
be a Borel measurable non-negative ($w_0\geq0$ a.e.\!)  
weight function which satisfies \eqref{a6}.
Then for all $f\in\calL$ (i.e. for all $f$ of the form \eqref{a5}) 
and for a.e.\  $\theta\in(0,2\pi)$, the limits in \eqref{a3}
exist in the norm of $\calK$. 
Further, there exists a non-trivial Borel measurable non-negative 
weight function $w_1:\bbT\to\calB(\calK)$, which satisfies
$$
\int_0^{2\pi}(w_1(e^{it})\chi,\chi)\frac{dt}{2\pi} \leq \norm{\chi}^2, 
\quad \forall \chi\in\calK,
$$
and there exist contractions (i.e. operators of norm $\leq1$) 
$X$, $Y_+$, $Y_-$, 
acting from $L^2(w_0)$ to $L^2(w_1)$, 
such that the weighted 
Hardy projections $P_\pm^{(w_0)}$ can be represented as
\begin{equation}
P_\pm^{(w_0)}=\pm\frac{i}{2}(X-Y_\pm).
\label{a7a}
\end{equation}
In particular, 
$$
P_\pm^{(w_0)}: 
L^2(w_0)\to L^2(w_1)
$$
are contractions. 
\end{theorem}
Let us discuss this result. 

1.
It is easy to see that the sum $P_+^{(w_0)}+P_-^{(w_0)}$ is simply the operator 
of multiplication by $w_0$:
$$
(P_+^{(w_0)}f)(e^{i\theta})+(P_-^{(w_0)}f)(e^{i\theta})
=
w_0(e^{i\theta})f(e^{i\theta}).
$$
By \eqref{a7a}, it follows that this operator of multiplication has norm $\leq1$.
From this it follows that 
\begin{equation}
w_0(e^{i\theta})^{1/2}w_1(e^{i\theta})w_0(e^{i\theta})^{1/2}\leq 1 
\label{a9}
\end{equation}
for a.e.\  $\theta\in(0,2\pi)$; see the end of Section~\ref{sec.d} for the details
of this argument.

2. In fact, more than \eqref{a9} is true; we note without proof that 
the boundedness of $P_\pm^{(w_0)}$ implies that 
$$
(P_r*w_0)^{1/2}(P_r*w_1)(P_r*w_0)^{1/2}\leq C
$$
for all $r<1$ 
with some constant $C$; here $P_r*w_{0,1}$ is the convolution
with the Poisson kernel
\begin{equation}
P_r(\theta)=\frac{1-r^2}{1+r^2-2r\cos\theta}.
\label{a9a}
\end{equation}

3.
Of course, the boundedness of $P_\pm^{(w_0)}$ implies that the 
weighted Hilbert transform 
$$
(H^{(w_0)} f)(e^{i\theta})
=
\lim_{r\to1}
\int_{0}^{2\pi}
w_0(e^{it})
\frac{2\sin(\theta-t)}{1+r^2-2r \cos(\theta-t)} f(e^{it})\frac{dt}{2\pi}
$$
is a bounded map from $L^2(w_0)$ to $L^2(w_1)$.

4.
If $p=1$, the weight $w_1$ can be chosen to satisfy 
\begin{equation}
\int_0^{2\pi}
\norm{w_1(e^{it})}_1\frac{dt}{2\pi}<\infty;
\label{a10}
\end{equation}
see the end of Section~\ref{sec.d}.

5.
The weight function $w_1$ constructed in the Koosis theorem is non-degenerate 
in the sense that $\log w_1\in L^1(\bbT)$. 
The weight function $w_1$ that we construct in Theorem~\ref{th.a1} 
is also non-degenerate in the following sense. 
One has 
$$
w_0(\mu)
=
D_0^+(\mu)^* w_1(\mu) D_0^+(\mu), 
\quad 
\text{ a.e.\  $\mu\in\bbT$,}
$$
where $D_0^+$ is an operator valued function to be constructed below
(see \eqref{a25}). The function $D_0^+$ satisfies
$\norm{D_0^+(\cdot)}_\calB\in L^{1,\infty}(\bbT)$ and 
$D_0^+(\mu)$ has a bounded inverse for a.e.\  $\mu\in\bbT$.
In particular, 
\begin{equation}
\rank w_0(\mu)=\rank w_1(\mu), 
\quad 
\text{ a.e.\  $\mu\in\bbT$,}
\label{a12}
\end{equation}
and 
\begin{equation}
\norm{w_1(\mu)}_\calB
\geq 
\frac{\norm{w_0(\mu)}_\calB}{\norm{D_0^+(\mu)}_\calB^2}, 
\quad 
\text{ a.e.\  $\mu\in\bbT$.}
\label{a30}
\end{equation}
By \eqref{a30}, we have
$$
\log \norm{w_1(\mu)}_\calB
\geq
\log \norm{w_0(\mu)}_\calB
-
2\log^+\norm{D_0^+(\mu)}_\calB,
$$
and $\norm{D_0^+(\cdot)}_\calB\in L^{1,\infty}(\bbT)$ 
implies
$\log^+\norm{D_0^+(\cdot)}_\calB\in L^p(\bbT)$ for all $p<\infty$.

\subsection{The outline of the proof}\label{sec.a3}
We consider the absolutely continuous (a.c.)  operator valued measure on $\bbT$ given by
\begin{equation}
d\nu_0(e^{i\theta})=w_0(e^{i\theta})\frac{d\theta}{2\pi}.
\label{a13}
\end{equation}
For this measure $\nu_0$, we exhibit (see Lemma~\ref{lma.b1}) 
a Hilbert space $\calH$, a unitary operator $U_0$ in $\calH$
and a contraction $G:\calH\to\calK$ such that 
\begin{equation}
\nu_0(\delta)=GE_{U_0}(\delta)G^*, \quad \delta\subset\bbT,
\label{a14}
\end{equation}
where $E_{U_0}$ is the projection-valued 
spectral measure of $U_0$, 
and  $\delta\subset\bbT$ is any Borel set.  
Next, 
we  construct (see \eqref{b3}) a unitary operator $U_1$ in $\calH$ 
such that the identities
\begin{align}
(\alpha+\psi_0(z))
(\alpha-\psi_1(z))
&=I,
\label{a15}
\\
(\alpha-\psi_1(z))
(\alpha+\psi_0(z))
&=I,
\label{a16}
\end{align}
hold true for all $\abs{z}\not=1$; 
here $\alpha$ is the auxiliary 
bounded self-adjoint operator given by 
\begin{equation}
\alpha=\sqrt{I-(GG^*)^2},
\label{a17}
\end{equation}
and 
\begin{equation}
\psi_j(z)=i G\frac{U_j+z}{U_j-z} G^*.
\label{a17a}
\end{equation}
Further, similarly to \eqref{a14}, we set
\begin{equation}
\nu_1(\delta)=GE_{U_1}(\delta)G^*, 
\quad \delta\subset \bbT.
\label{a18}
\end{equation}
We will be able to prove (in Lemma~\ref{lma.c2}) that the a.c. part of the measure $\nu_1$
can be represented as
$$
d\nu^\ac_1(e^{i\theta})=w_1(e^{i\theta})\frac{d\theta}{2\pi}
$$
with some operator valued non-negative weight function $w_1$. 
Note that this is not automatic: the Radon-Nikodym theorem for operator
valued measures in general fails; to see this, consider the spectral measure
of a self-adjoint or unitary operator with a non-trivial a.c. component.

Key to our construction is the connection between the weighted 
Hardy projections $P_\pm^{(w_0)}$
and certain operators appearing in scattering theory for 
the pair $U_0$, $U_1$. We use the formalism suggested by de Branges
\cite{deB} with some simplifications due to Kuroda \cite{Ku}.
This formalism makes use of the weighted Hilbert spaces 
$L^2(\nu_j)$, $j=0,1$ of $\calK$-valued functions on $\bbT$.
They are defined, similarly to \eqref{a4b}, starting from the quasi-norm
$$
\norm{f}_{L^2(\nu_j)}^2
=
\int_0^{2\pi} d(\nu_j(e^{it})f(e^{it}),f(e^{it}))
$$
on the set $\calL$, by taking a quotient and then a closure. 
We note that $\nu_0=\nu^\ac_0$ and 
\begin{equation}
L^2(\nu_1)\subset L^2(\nu_1^\ac)
\quad \text{ and } \quad
\norm{f}_{L^2(\nu_1^\ac)}
\leq
\norm{f}_{L^2(\nu_1)}.
\label{a20}
\end{equation}

Following de Branges, we define some auxiliary bounded operators $X$, $Y_+$ and $Y_-$
acting from $L^2(\nu_0)$ to $L^2(\nu_1^{\ac})$.
First we denote (cf. \eqref{a15}, \eqref{a16})
\begin{equation}
D_0(z)
=
\alpha+\psi_0(z), 
\quad
D_1(z)
=
-\alpha+\psi_1(z).
\label{a21}
\end{equation}
By \eqref{a15}, \eqref{a16} we have
\begin{equation}
D_0(z)D_1(z)=D_1(z)D_0(z)=-I, \quad \abs{z}\not=1.
\label{a22}
\end{equation}
Let 
\begin{equation}
X: L^2(\nu_0)\to L^2(\nu_1)
\label{a23}
\end{equation}
be the linear operator, defined on the dense set $\calL$ by 
\begin{equation}
(Xf)(\mu)
=
\sum_i (\mu-z_i)^{-1} D_0(z_i)\chi_i, 
\quad
f(\mu)=\sum_i (\mu-z_i)^{-1}\chi_i.
\label{a24}
\end{equation}
It turns out (see Lemma~\ref{lma.b2}) that $X$ is a unitary operator
between the spaces \eqref{a23}. Moreover, this is
true for \emph{any} operators $U_0$, $U_1$, $G$, $\alpha$, 
related by \eqref{a15}--\eqref{a17a}; assumption \eqref{a6} is 
not relevant here.
This fact is part of  de Branges' construction \cite{deB}.
Bearing in mind the embedding \eqref{a20}, we see that
$X$ is a contraction as a map from $L^2(\nu_0)$ 
to $L^2(\nu_1^\ac)$.

Further, by the spectral theorem for the unitary operator $U_0$, 
we have
\begin{equation}
\psi_0(z)
=
i\int_0^{2\pi}\frac{e^{it}+z}{e^{it}-z}d\nu_0(e^{it})
=
i\int_0^{2\pi}\frac{e^{it}+z}{e^{it}-z} w_0(e^{it})\frac{dt}{2\pi}.
\label{a24a}
\end{equation}
Thus, $\psi_0$ is the Cauchy transform of $w_0$. Using  assumption 
\eqref{a6} on the weight $w_0$ and the \emph{UMD property} 
(see e.g. \cite{deF}) of the space $\Sch_p$, $1<p<\infty$, 
we check (in Lemma~\ref{lma.c1}) that the limits 
\begin{equation}
D_0^\pm(e^{i\theta})
=
\lim_{r\to1\pm0}
D_0(re^{i\theta}),
\quad
D_1^\pm(e^{i\theta})
=
\lim_{r\to1\pm0}
D_1(re^{i\theta})
\label{a25} 
\end{equation}
exist for a.e.\  $\theta\in(0,2\pi)$ in the operator norm.
For $p=1$, this was proven in \cite{deB}; for $p>1$, this fact is 
borrowed from from our related work \cite{PuV}. 

Again following de Branges, we consider the operators 
\begin{equation}
Y_\pm: f(\mu)\mapsto D_0^{\pm}(\mu)f(\mu), 
\quad
\mu\in\bbT,
\label{a26}
\end{equation}
defined initially on the set $\calL$, and show that $Y_\pm$ extend
as isometric operators 
$$
Y_\pm: L^2(\nu_0)\to L^2(\nu_1^\ac).
$$
Finally, a simple calculation (see Section~\ref{sec.d})  shows that 
$P_\pm^{(w_0)}$, $X$, $Y_\pm$ are related by \eqref{a7a}.
We note that $Y_\pm$ are unitarily equivalent to the wave operators
$W_\pm(U_1,U_0)$ (see \cite{Ku}), although we will not need this fact.

\section{Identities \eqref{a15}, \eqref{a16} and the map $X$}\label{sec.b}

\subsection{The construction of $G$, $U_0$, $U_1$, $\alpha$ }\label{sec.b1}

Let $\calH$ be the Hilbert space of all Borel measurable $\calK$-valued 
functions on $\bbT$ with the norm
$$
\norm{f}_\calH^2
=
\int_0^{2\pi} \norm{f(e^{it})}^2 \frac{dt}{2\pi}. 
$$
Let $U_0$ be the operator of multiplication by $e^{it}$ in $\calH$. 
Let $G:\calH\to\calK$ be defined by 
$$
Gf=\int_0^{2\pi} w_0(e^{it})^{1/2} f(e^{it})\frac{dt}{2\pi}. 
$$
Then our assumption \eqref{a6} implies that $G$ is a contraction:
\begin{multline*}
\norm{Gf}
\leq
\left(\int_0^{2\pi}\norm{w_0(e^{it})^{1/2}}_\calB^2\frac{dt}{2\pi}\right)^{1/2}
\left(\int_0^{2\pi}\norm{f(e^{it})}^2\frac{dt}{2\pi}\right)^{1/2}
\\
=
\left(\int_0^{2\pi}\norm{w_0(e^{it})}_\calB\frac{dt}{2\pi}\right)^{1/2}
\norm{f}_\calH
\leq
\norm{f}_\calH.
\end{multline*}
It is clear that setting 
$\nu_0(\delta)=GE_{U_0}(\delta)G^*$ (see \eqref{a14}) yields \eqref{a13}. 
Next, let 
$$
\Theta =2\sin^{-1}(G^*G); 
$$
thus, $\Theta$ is a bounded self-adjoint operator in $\calH$ with 
$\sigma(\Theta)\subset [0,\pi)$ and 
\begin{equation}
G^*G=\sin(\tfrac12 \Theta).
\label{b2}
\end{equation}
Set 
\begin{equation}
U_1=\exp(\tfrac{i}2 \Theta)U_0\exp(\tfrac{i}2 \Theta)
\label{b3}
\end{equation}
and let $\alpha$ be defined by \eqref{a17}.

\begin{lemma}\label{lma.b1}
Let $U_0$, $U_1$, $G$, $\alpha$ be as described above. 
Then identities \eqref{a15}, \eqref{a16} hold true. 
The measure $\nu_1$, defined by \eqref{a18}, satisfies
$\nu_1(\bbT)=\nu_0(\bbT)$ and 
\begin{equation}
\norm{\nu_1(\bbT)}\leq 1. 
\label{b3a}
\end{equation}
\end{lemma}
\begin{proof}
Denote 
$$
\beta=\sqrt{I-(G^*G)^2};
$$
clearly, we have 
\begin{equation}
\alpha G=G\beta.
\label{b5}
\end{equation}
Comparing \eqref{b2} and the definition of $\beta$, we find that
$$
\beta=\cos(\tfrac12\Theta).
$$
Using this and a little algebra, we obtain
$$
U_1G^*G+G^*GU_0+i(U_1\beta-\beta U_0)=0.
$$
From here by straightforward manipulation we obtain the identity
$$
(U_1-z)G^*G(U_0-z)
+
i\bigl((U_1+z)\beta(U_0-z)-(U_1-z)\beta(U_0+z)\bigr)
-
(U_1+z)G^*G(U_0+z)=0
$$
for any $z\in\bbC$. Taking $\abs{z}\not=1$ and multiplying 
by $(U_1-z)^{-1}$ on the left and by $(U_0-z)^{-1}$ on the right, 
we get
$$
G^*G
+
i\left(
\frac{U_1+z}{U_1-z}\beta-\beta\frac{U_0+z}{U_0-z}
\right)
-
\frac{U_1+z}{U_1-z}G^*G\frac{U_0+z}{U_0-z}
=0.
$$
Multiplying this by $G$ on the left and by $G^*$ on the right
and using that (by \eqref{a17})
$$
(GG^*)^2=I-\alpha^2,
$$
we obtain 
$$
-\alpha^2
+
iG\frac{U_1+z}{U_1-z}\beta G^*
-
i G\beta\frac{U_0+z}{U_0-z} G^*
-
G\frac{U_1+z}{U_1-z}G^* G\frac{U_0+z}{U_0-z} G^*
=
-I.
$$
Finally, using \eqref{b5}, this transforms into \eqref{a16}. 
The relation \eqref{a15} is obtained by taking adjoints in \eqref{a16} and 
changing $z$ to $\overline{z}^{-1}$.
By \eqref{a14}, \eqref{a18}, we have $\nu_0(\bbT)=\nu_1(\bbT)=GG^*$. 
The estimate \eqref{b3a} follows from 
the inequality $\norm{G}\leq 1$. 
\end{proof}

\begin{remark}
In fact, the construction of \cite{deB,Ku} allows for a whole family of 
possible choices for operators $G$, $U_0$, $U_1$, suitable for our
argument. For simplicity, we have chosen only one representative of this 
family. 
\end{remark}

\begin{remark}
In order to clarify the ideas behind Lemma~\ref{lma.b1}, let us sketch 
the analogous argument for the case of the weights $w_0$, $w_1$ 
on the real line. 
In this case the construction naturally leads to self-adjoint (rather than
unitary) operators and the algebra is somewhat more 
transparent. Let a non-negative weight $w_0:\bbR\to\calB(\calK)$ 
satisfy 
$$
\int_\bbR\norm{w_0(t)}_\bbR dt<\infty.
$$
Let $\calH_\bbR$ be the $L^2$ space of $\calK$-valued functions
on $\bbR$ with the norm
$$
\norm{f}_{\calH_\bbR}^2
=
\int_\bbR \norm{f(t)}^2dt.
$$
Let $A_0$ be the operator of multiplication by the independent variable
$t$ in $\calH_\bbR$ and let $G_\bbR:\calH_\bbR\to\calK$  be given by 
$$
G_\bbR f=\int_\bbR w_0(t)^{1/2} f(t)dt.
$$
We set $A_1=A_0+G_\bbR^*G_\bbR$. Then from the standard resolvent identity we get
\begin{multline*}
(I+G_\bbR(A_0-z)^{-1}G_\bbR^*)
(I-G_\bbR(A_1-z)^{-1}G_\bbR^*)
\\
=
(I-G_\bbR(A_1-z)^{-1}G_\bbR^*)
(I+G_\bbR(A_0-z)^{-1}G_\bbR^*)
=I;
\end{multline*}
this is the analogue of \eqref{a15}, \eqref{a16}. 
One sets 
$$
\nu_j^\bbR(\delta)=G_\bbR E_{A_j}(\delta)G_\bbR^*, 
\quad
j=0,1,
\quad
\delta\subset \bbR,
$$
and the rest of the construction is very similar to the case of measures on $\bbT$. 
\end{remark}

\subsection{The map $X$}\label{sec.b2}
Let the map $X$ be defined by \eqref{a23}, \eqref{a24}.
\begin{lemma}\label{lma.b2}
The map $X$ is unitary between the spaces $L^2(\nu_0)$ and $L^2(\nu_1)$. 
\end{lemma}
\begin{proof}
For $j=0,1$, the functions $\psi_j$ (see \eqref{a17a}) can be expressed as 
\begin{equation}
\psi_j(z)=i\int_0^{2\pi}\frac{e^{it}+z}{e^{it}-z}d\nu_j(e^{it}).
\label{b6}
\end{equation}
We note two identities for $\psi_j$:
\begin{gather}
\frac{\psi_j(z_1)-\psi_j(z_2)^*}{z_1-\overline{z_2}^{-1}}
=
2i\int_0^{2\pi} 
\frac{e^{it}}{(e^{it}-z_1)(e^{it}-\overline{z_2}^{-1})} d\nu_j(e^{it}),
\label{b7}
\\
\psi_j(z)^*=\psi_j(\overline{z}^{-1}).
\label{b8}
\end{gather}
Next, using \eqref{a21}, \eqref{a22}, we have for $\abs{z_{1,2}}\not=1$:
\begin{multline*}
\psi_0(z_1)-\psi_0(z_2)^*
=
D_0(z_1)-D_0(z_2)^*
\\
=
-D_0(z_2)^*D_1(z_2)^*D_0(z_1)
+
D_0(z_2)^*D_1(z_1)D_0(z_1)
\\
=
D_0(z_2)^*(-\psi_1(z_2)^*+\psi_1(z_2))D_0(z_1).
\end{multline*}
Combining this with \eqref{b7}, \eqref{b8}, we get
\begin{equation}
\int_0^{2\pi}
\frac{d\nu_0(e^{it})}{(e^{-it}-\overline{z_2})(e^{it}-z_1)}
=
D_0(z_2)^*
\int_0^{2\pi}
\frac{d\nu_1(e^{it})}{(e^{-it}-\overline{z_2})(e^{it}-z_1)}
D_0(z_1).
\label{b9}
\end{equation}
Now let 
\begin{equation}
f_1(\mu)=(\mu-z_1)^{-1}\chi_1,
\quad
f_2(\mu)=(\mu-z_2)^{-1}\chi_2,
\quad
\mu\in\bbT,
\label{b10}
\end{equation}
where $\abs{z_{1,2}}\not=1$ and $\chi_{1,2}\in\calK$. 
Then from \eqref{b9} we get
$$
(f_1,f_2)_{L^2(\nu_0)}
=
(Xf_1,Xf_2)_{L^2(\nu_1)}.
$$
This extends to all $f_1,f_2\in\calL$. 
It follows that $X$ is an isometry.
By considering an operator $X_1$ defined in a similar way with $D_1$ instead of 
$D_0$, and using \eqref{a22}, we obtain $XX_1=-I$, hence $X$ is a surjection. 
Thus, $X$ is a unitary operator. 
\end{proof}

\section{The boundary values of $D_0$ and $D_1$}\label{sec.c}

\subsection{Existence of boundary values of $D_0$ and $D_1$}\label{sec.c1}

\begin{lemma}\label{lma.c1}
The limits $D_0^\pm(e^{i\theta})$, 
$D_1^\pm(e^{i\theta})$ (see \eqref{a25}) exist for a.e.\  $\theta\in(0,2\pi)$
in the operator norm.
\end{lemma}
For $p=1$, this was proven in \cite{deB}.
\begin{proof}
1. First we consider the limits $D_0^\pm$. 
We have
$$
D_0(z)=\alpha+\psi_0(z),
$$
where $\psi_0(z)$ is given by \eqref{b6}. 
Thus, it suffices to consider the limits of $\psi_0$. 
By \eqref{b8}, it suffices to consider the limits as $z$ 
approaches the unit circle from inside the unit disk. 
Without loss of generality assume $p>1$ in \eqref{a6}. 
In fact, we will prove the existence of the non-tangential limits
$$
\lim_{\genfrac{}{}{0pt}{}{z\to e^{i\theta}}{z\in S_\theta}} \psi_0(z)
$$
in the norm of $\Sch_p$. Here $S_\theta$ is the appropriate sector of opening $\pi/2$ with the vertex at $e^{i\theta}$
(see e.g. \cite[Section VIII:C3]{Koosis2}).
The argument below is presented in more detail in our related work \cite{PuV}.

The function $\psi_0$ is the Cauchy transform of the weight function $w_0$ 
(see \eqref{a24a}). 
Consider the non-tangential maximal function 
$$
(T w_0)(e^{i\theta})
=
\sup \left\{\Norm{\psi_0(z)}_p: z\in S_\theta\right\}.
$$
The key fact is that for $1<p<\infty$, 
the Banach space $\Sch_p$ possesses the UMD property, see \cite{deF}; 
that is, the Hilbert transform and many other integral transforms are bounded as operators
in $L^2$ spaces of $\Sch_p$-valued functions. 
Using this, one can prove that the (non-linear) operator $T$ is of the weak 1-1 type, i.e. 
$T w_0$ belongs to the weak $L^{1,\infty}(\bbT)$ class. 

Next, using this fact and repeating the classical construction of Privalov's uniqueness theorem
(see e.g. \cite[Section III:D]{Koosis2}), for any $\eps>0$ one constructs 
a simply connected domain $\calD$ in the unit disk such that $\norm{\psi_0}_p$ is bounded in $\calD$ 
and the boundary of $\calD$ contains the unit circle $\bbT$ up to a set of measure $\eps$. 
Let $\varphi$ be a conformal map of the unit disk onto $\calD$. Then $F(z)=\psi_0(\varphi(z))$
is a bounded $\Sch_p$-valued analytic function on the unit disk. 
By standard results on Banach space valued analytic functions (see e.g. \cite{Bu}), 
$F(z)$ attains non-tangential boundary values in $\Sch_p$ norm a.e.\  on the unit circle. 
It follows that the function $\psi_0$ attains non-tangential boundary values 
in $\Sch_p$ norm on the unit circle minus a set of measure $\eps$. 
Sending $\eps\to0$, one obtains the desired result. 

2. 
Let us consider the limits of $D_1$. 
Since $D_1(z)=-D_0(z)^{-1}$, it suffices to prove that the limiting operators $D_0^\pm(e^{i\theta})$
have bounded inverses for a.e.\  $\theta$. 
We do this by employing an argument from \cite{Yafaev}. 
We have
$$
D_0(z)=D_0(0)\left(I+D_0(0)^{-1}(D_0(z)-D_0(0))\right), 
$$
and therefore it suffices to check that the operators
\begin{equation}
I+D_0(0)^{-1}(D_0^\pm(e^{i\theta})-D_0(0))
\label{c1}
\end{equation}
have a bounded inverse for a.e.\  $\theta$. 
By \eqref{a6}, we have $\psi_0(z)\in\Sch_p$ for all $\abs{z}\not=1$. 
Let $q\geq p$ be any integer; consider the regularised determinant
$$
d(z)=\Det_q(I+D_0(0)^{-1}(\psi_0(z)-\psi_0(0))).
$$
The functional $A\mapsto\Det_q(I+A)$ is continuous (in fact, analytic) on $\Sch_q$. 
Thus, $d(z)$ is analytic in $z$ and by the previous step of the proof, $d(z)$ has 
non-tangential boundary values a.e.\  on the unit circle. 
Applying Privalov's uniqueness theorem, we obtain that these boundary values
are non-zero a.e.\  on the unit circle. Now since $\Det_q(I+A)\not=0$ if and only
if $I+A$ has a bounded inverse, we conclude that the operators \eqref{c1} 
have bounded inverses for a.e.\  $\theta$. 
\end{proof}

\subsection{The a.c. part of $\nu_1$}\label{sec.c2}
Taking $z_1=z_2=re^{i\theta}$ in \eqref{b7}, 
one obtains
\begin{equation}
\psi_j(re^{i\theta})-\psi_j(re^{i\theta})^*
=
2i\int_0^{2\pi} P_r(\theta-t)d\nu_j(e^{it}),
\label{c2}
\end{equation}
where $P_r$ is the Poisson kernel \eqref{a9a} on $\bbT$. 
From the existence of the boundary 
values of $\psi_j$ on $\bbT$ (see Lemma~\ref{lma.c1})
it folows that the r.h.s. of \eqref{c2} attains a limit 
(in the operator norm) as $r\to1$ for a.e.\  $\theta\in(0,2\pi)$. 
Of course, by the definition \eqref{a13} of $\nu_0$ we have
\begin{equation}
w_0(e^{i\theta})
=
\lim_{r\to1} 
\int_0^{2\pi} P_r(\theta-t)d\nu_0(e^{it})
\label{c7}
\end{equation}
for a.e.\  $\theta$. 
Similarly, we \emph{define} the weight function $w_1$
by 
\begin{equation}
w_1(e^{i\theta})= \lim_{r\to1} 
\int_0^{2\pi} P_r(\theta-t)d\nu_1(e^{it})
\label{c3}
\end{equation}
for a.e.\  $\theta$. 
In Lemmas~\ref{lma.c2} and \ref{lma.c3}, we follow de Branges' work \cite{deB}.

\begin{lemma}\label{lma.c2}
The a.c. part of the measure $\nu_1$ is given by 
\begin{equation}
d\nu_1^\ac(e^{i\theta})=w_1(e^{i\theta})\frac{d\theta}{2\pi}, 
\quad
\text{ a.e.\  $\theta\in(0,2\pi)$.}
\label{c4}
\end{equation}
\end{lemma}
\begin{proof}
Of course, in the scalar case $\dim\calK<\infty$ formula \eqref{c4} follows
directly from \eqref{c3}; the point here is to consider the general case. 
Let $\chi_1,\chi_2\in\calK$; consider the scalar (complex-valued)
measure $(\nu_1(\cdot)\chi_1,\chi_2)$. 
If $\nu_1^\ac$ and $\nu_1^\sing$ 
are the a.c. and the singular parts of $\nu_1$ with respect to the Lebesgue
measure on $\bbT$, then 
$$
(\nu_1(\cdot)\chi_1,\chi_2)
=
(\nu_1^\ac(\cdot)\chi_1,\chi_2)
+
(\nu_1^\sing(\cdot)\chi_1,\chi_2)
$$
gives the unique decomposition of the scalar measure 
$(\nu_1(\cdot)\chi_1,\chi_2)$ into the a.c. and singular parts. 
By the scalar theory, we have
$$
\lim_{r\to1}\int_0^{2\pi} P_r(\theta-t)d(\nu_1^\sing(e^{it})\chi_1,\chi_2)=0
$$
for a.e.\  $\theta$. Thus, using \eqref{c3}, we obtain
\begin{equation}
(w_1(e^{i\theta})\chi_1,\chi_2)
= 
\lim_{r\to1} 
\int_0^{2\pi} P_r(\theta-t)d(\nu_1^\ac(e^{it})\chi_1,\chi_2).
\label{c5}
\end{equation}
Now take $f_1$, $f_2$ as in \eqref{b10}; multiplying \eqref{c5} 
by $(e^{i\theta}-z_1)^{-1}(e^{-i\theta}-\overline{z_2})^{-1}$ and integrating,
we get
$$
\int_0^{2\pi}
(w_1(e^{i\theta})f_1(e^{i\theta}), f_2(e^{i\theta}))\frac{d\theta}{2\pi}
=
\int_0^{2\pi}
d(\nu_1^\ac(e^{i\theta})f_1(e^{i\theta}),f_2(e^{i\theta})).
$$
By linearity, this extends to all $f_1,f_2\in\calL$. This yields \eqref{c4}.
\end{proof}

\subsection{The operators $Y_\pm$}\label{sec.c3}
Next, we consider the operators $Y_\pm$ of multiplication
by $D_0^\pm$, see \eqref{a26}.

\begin{lemma}\label{lma.c3}
The operators $Y_\pm$ are unitary maps from $L^2(\nu_0)$ 
to $L^2(\nu_1^\ac)$. 
\end{lemma}
\begin{proof}
Taking $z_1=z_2=re^{i\theta}$ in \eqref{b9}, we obtain
$$
\int_0^{2\pi}
P_r(\theta-t)d\nu_0(e^{it})
=
D_0(re^{i\theta})^*
\int_0^{2\pi} P_r(\theta-t)d\nu_1(e^{it}) \ 
D_0(re^{i\theta}).
$$
Taking $r\to1\pm0$ and using Lemma~\ref{lma.c2}, we get
$$
w_0(\mu)
=
D_0^\pm(\mu)^*w_1(\mu)D_0^\pm(\mu),
\quad
\text{ a.e.\  $\mu\in\bbT$.}
$$
This shows that $Y_\pm$ are isometries. 
Considering the operators of multiplication by the boundary values
of $D_1$ and using the identity \eqref{a22}, it is easy to prove that $Y_\pm$ 
are surjections, so they 
are unitary operators. 
\end{proof}

\section{The proof of Theorem~\ref{th.a1}}\label{sec.d}

\begin{proof}[Proof of Theorem~\ref{th.a1}]

The weight function $w_1$ has been defined by \eqref{c3}. 
By construction, it is non-negative. 
It is Borel measurable as a pointwise norm limit of continuous weight functions. 
Let us prove that 
the limits in \eqref{a1} exist in $\calK$ and 
\begin{equation}
(P^{(w_0)}_\pm f)(e^{i\theta})
=
\pm\frac{i}{2}((Xf)(e^{i\theta})-(Y_\pm f)(e^{i\theta}))
\label{d1}
\end{equation}
for a.e. $\theta$.
Take $f(\mu)=(\mu-z)^{-1}\chi$, $\chi\in\calK$, $\abs{z}\not=1$. 
We have
$$
D_0(z)=\alpha+i\int_0^{2\pi}d\nu_0(e^{it})\frac{e^{it}+z}{e^{it}-z}
=
\alpha
-
i\int_0^{2\pi}d\nu_0(e^{it})
+
2i
\int_0^{2\pi}d\nu_0(e^{it})\frac{e^{it}}{e^{it}-z},
$$
and therefore, by the definition \eqref{a24} of $X$, 
$$
(Xf)(e^{i\theta})
=
\left(\alpha
-
i\int_0^{2\pi}d\nu_0(e^{it})\right)f(e^{i\theta})
+
2i
\int_0^{2\pi}d\nu_0(e^{it})\frac{e^{it}}{(e^{i\theta}-z)(e^{it}-z)}\chi.
$$
For the second term in the above sum, we have
\begin{multline*}
\int_0^{2\pi}
d\nu_0(e^{it})\frac{e^{it}}{(e^{i\theta}-z)(e^{it}-z)}\chi
=
\int_0^{2\pi}
d\nu_0(e^{it})\frac{f(e^{i\theta})-f(e^{it})}{e^{it}-e^{i\theta}}e^{it}
\\
=
\lim_{r\to1}
\left\{
\int_0^{2\pi}
d\nu_0(e^{it})\frac{f(e^{i\theta})}{1-re^{i(\theta-t)}}
-
\int_0^{2\pi}
d\nu_0(e^{it})\frac{f(e^{it})}{1-re^{i(\theta-t)}}
\right\},
\end{multline*}
where the limits exist in the norm of $\calK$.
Putting this together, after a little algebra we get
\begin{equation}
(Xf)(e^{i\theta})
=
\lim_{r\to1}
\left\{D_0(re^{i\theta})f(e^{i\theta})
-2i
\int_0^{2\pi}
d\nu_0(e^{it})\frac{f(e^{it})}{1-re^{i(\theta-t)}}\right\}.
\label{d3}
\end{equation}
By Lemma~\ref{lma.c1} the limits
$$
\lim_{r\to\pm1}
D_0(re^{i\theta})f(e^{i\theta})
$$
exist in the norm of $\calK$. 
Thus, the limits of the integral in \eqref{d3} also exist. 
Recalling the definition \eqref{a1} of $P_\pm^{(w_0)}$, we obtain \eqref{d1}.
\end{proof}

\begin{proof}[Proof of \eqref{a9}]
By \eqref{c2} and \eqref{c7}, we have
\begin{multline*}
w_0(e^{i\theta})
=
\lim_{r\to1}\int_0^{2\pi} P_r(\theta-t)d\nu_0(e^{it})
\\
=
\frac1{2i}\lim_{r\to1} (\psi_0(re^{i\theta})-\psi_0(\tfrac1r e^{i\theta}))
=
\frac1{2i}\lim_{r\to1} (D_0(re^{i\theta})-D_0(\tfrac1r e^{i\theta})).
\end{multline*}
Thus, if we denote by $Y_0$ the operator of multiplication 
by $w_0(e^{i\theta})$, acting from $L^2(\nu_0)$ to $L^2(\nu_1^\ac)$, 
we obtain 
$$
Y_0=\frac1{2i}(Y_+-Y_-), 
$$
and therefore $\norm{Y_0}\leq 1$. This yields
$$
\int_0^{2\pi} (w_1(e^{i\theta})w_0(e^{i\theta})f(e^{i\theta}), w_0(e^{i\theta})f(e^{i\theta}))\frac{d\theta}{2\pi}
\leq
\int_0^{2\pi} (w_0(e^{i\theta})f(e^{i\theta}),f(e^{i\theta}))\frac{d\theta}{2\pi}, 
$$
which implies \eqref{a9}.
\end{proof}

\begin{proof}[Proof of \eqref{a10}]
Suppose $p=1$. Then 
$$
\Tr (GG^*)
=
\int_0^{2\pi} \Tr( w_0(e^{i\theta}))\frac{d\theta}{2\pi}
\leq 1, 
$$
hence $G$ is Hilbert-Schmidt. 
Then 
$$
\int_0^{2\pi} \norm{w_1(e^{i\theta})}_1\frac{d\theta}{2\pi}
=
\int_0^{2\pi} \Tr(w_1(e^{i\theta}))\frac{d\theta}{2\pi}
=
\Tr
(\nu_1^\ac (\bbT))
\leq
\Tr(\nu_1(\bbT))
=
\Tr(GG^*)\leq 1,
$$
i.e. $\norm{w_1(\cdot)}_1\in L^1(\bbT)$. 
\end{proof}

\end{document}